\documentclass[12pt,a4paper,reqno]{amsart}
\usepackage[activeacute,english]{babel}

\usepackage[applemac]{inputenc}
\usepackage{amsfonts}
\usepackage{amsmath,amsthm,bm}
\usepackage{amssymb}
\usepackage{graphics}
\usepackage{tikz}
\usepackage{float}

\usepackage{hyperref}
\usepackage{cleveref}

\DeclareRobustCommand{\SkipTocEntry}[4]{}

\newtheorem{theorem}{Theorem}[section]
\newtheorem{lemma}[theorem]{Lemma}

\newtheorem{proposition}[theorem]{Proposition}

\newtheorem{remark}[theorem]{Remark}

\newtheorem{definition}[theorem]{Definition}

\numberwithin{equation}{section}

\allowdisplaybreaks[1]
\title{Distributional Solutions of the Damped Wave Equation}

\author{Marc Nualart}
\address{Departament de Matem\`{a}tiques,
Universitat Polit\`{e}cnica de Catalunya
C. Pau Gargallo, 14
08028 Barcelona}
\email{marc.nualart@upc.edu}

\date{March 13, 2020}


\begin{document}

%%%%%%%%%

\begin{abstract}
This work presents results on solutions of the one-dimensional damped wave equation, also called telegrapher's equation, when the initial conditions are general distributions, not only functions. We make a complete deduction of its fundamental solutions, both for positive and negative times. To obtain them we use only self-similarity arguments and distributional calculus, making no use of Fourier or Laplace transforms.

We next use these fundamental solutions to prove both the existence and the uniqueness of solutions to the distributional initial value problem. As applications we recover the semigroup property for initial data in classical function spaces and also the probability distribution function for a recent financial model of evolution of prices.
\end{abstract}
\maketitle
\section{Introduction}

The one dimensional damped wave equation $u_{tt}+ku_t=c^2u_{xx}$ for $k,c>0$ has been vastly studied (see \cite{G&L, Duf} and \cite{Ock}, for example) and related to several important phenomena. These are, for example, the mechanical oscillations of a \textit{string with friction} \cite{Salsa}, the four examples considered in Section 2 of \cite{MW}, and, of more interest in the last decade, the \textit{persistent motion} in Movement Ecology \cite{MCB, Rossetto} and the probability density function for a price evolution model of a \textit{financial asset} \cite{Kol}.

The equation has been considered either in the whole real line or in an interval with boundary conditions. In both cases, many properties of its solutions, particularly of their decay on time, are known for initial conditions $u(x,0)=f(x)$ and $u_t(x,0)=g(x)$ belonging to different classes of functions, including distributions. In particular, in \cite{MW}, the attention is focussed to the case of the whole real line and $f(x)=\delta(x)$, the Dirac's delta function and $g(x)=0$, obtaining 
\begin{equation*}
u(x,t)=\frac{1}{2}e^{-\frac{k}{2}t}\left\lbrace \delta(x-ct)+\delta(x+ct)\right\rbrace 
+k\frac{e^{-\frac{k}{2}t}}{8c} \left[ I_0(\xi)+k\frac{I_1(\xi)}{2\xi}\right]H(ct-|x|),
\end{equation*}
where $I_0$, $I_1$ are Bessel functions, $H(x)$ is the Heaviside function and $$\xi=\frac{\sqrt{c^2t^2-x^2}}{2ct}.$$
This is a solution to be compared, in the context of random walks, to the Gaussian solution of the Diffusion Equation $u_t=Du_{xx}$ with $u(x,0)=\delta(x)$, namely $$u(x,t)=\dfrac{1}{\sqrt{4\pi Dt}} \exp{\left(-\dfrac{x^2}{4Dt}\right)}.$$
And in \cite{Kol}, the attention is focussed in the case $f(x)=\delta(x)$ and $g(x)=-c\delta'(x)$, a case that will be considered in detail in our Section 6 below.

In the whole real line setting, the usual methods often include Fourier and Laplace transforms. In fact, this is the approach of J. Masoliver in (\cite{M}, Appendix A), where the solution to the case $f(x)=\delta(x)$ and $g(x)=0$ is obtained with the use of the Laplace-Fourier Transform. In the same reference it is stressed that "\emph{although the solution has been known since a very long time ago, its derivation has remained quite obscure in the literature}".

In the present paper we will solve the equation in the whole real line both for positive and negative times, a case that has not always been studied in detail, when the initial conditions are general distributions. Furthermore, we will prove that this solution is unique, a result that we have not been able to find explicitly in the literature for general initial conditions. This clearly differs from the non-uniqueness phenomena that happen for the one-dimensional heat equation in the real line when one accepts solutions that can grow very fast near infinity.

Our results will be based on deducing first the fundamental solutions for the equations. We emphasize that we will do this without any use of Fourier or Laplace transforms, but only with self-similar arguments and distributional calculus. This will allow us not to be forced to restrict ourselves to only tempered distributions. We believe that this procedure is more complete and clear than in the previous literature.

The initial-value problem is
\begin{equation}\label{main,u}
\begin{array}{l}\left\{\begin{array}{cc}u_{tt}+ku_t=c^2u_{xx}\\u(0)=f, \quad u_t(0)=g&\end{array}\right.\\\end{array}
\end{equation}
which is equivalent to
\begin{equation}\label{main,v}
\begin{array}{l}\left\{\begin{array}{cc}v_{tt}=c^2v_{xx}+\frac{k}{4}v\\v(0)=f, \quad v_t(0)=g+\frac{k}{2}f&\end{array}\right.\\\end{array}
\end{equation}
for $v=e^{\frac{k}{2}t}u$.

In the following results, the meaning of solution is always in $\mathcal{D}'(\mathbb{R}\times\mathbb{R})$, but we will prove that the solutions belong to the space ${\mathcal C}^1(\mathbb{R},\mathcal{D}'(\mathbb{R}))$ of ${\mathcal C}^1$ functions on time with values on the space of distributions in $x\in\mathbb{R}$. This gives us a stronger regularity for the solution than just the $\mathcal{D}'(\mathbb{R}\times\mathbb{R})$ meaning.
\begin{theorem}\label{exs,psi}
Let us consider the function
\begin{equation*}
\psi(x,t)=sgn(t)\frac1{2c}I_0\left(2\alpha\sqrt{c^2t^2-x^2}\right){\mathcal X}_{\lbrack-c|t|,c|t|\rbrack}(x),
\end{equation*}
with $\alpha=\frac{k}{4c}$, $I_0$ the modified Bessel function of first kind and parameter 0, $sgn(t)$ the sign function, and ${\mathcal X}_\Omega$ the characteristic function of $\Omega$. \\
Then, $\psi(x,t)$ belongs to ${\mathcal C}^1(\mathbb{R},\mathcal{D}'(\mathbb{R}))$ and solves \eqref{main,v} in the sense of distributions when $f=0$ and $g=\delta$ the Dirac delta centered at $x=0$. Furthermore, its time-derivative turns out to be
\begin{equation*}
\psi_t(\cdot,t)=\frac12\left[\delta(\cdot-ct)+\delta(\cdot+ct)\right]+\alpha c|t|\frac{I_0'\left(2\alpha\sqrt{c^2t^2-\cdot^2}\right)}{\sqrt{c^2t^2-\cdot^2}}{\mathcal X}_{\left[-c|t|,c|t|\right]}(\cdot),
\end{equation*}
which is not a function but a distribution.
\end{theorem}
This solution $\psi$ will be called the fundamental solution of the problem. From $\psi$ and $\psi_t$ we are able to obtain the solutions when $f$ and $g$ are general distributions, as the next result shows.

\begin{theorem}\label{exs,f,g}
Given $f,g\in\mathcal{D}'(\mathbb{R})$ distributions, the solution $v\in{\mathcal C}^1(\mathbb{R},\mathcal{D}'(\mathbb{R}))$ to the initial value problem \eqref{main,v} in the sense of distributions is given by
\begin{equation*}
\left\langle v,\varphi\right\rangle := \left\langle f, \psi_t \ast\varphi \right\rangle +\left\langle g+\frac{k}{2}f,\psi\ast\varphi \right\rangle.
\end{equation*}
\end{theorem}
This solution is unique as stated next.
\begin{theorem}\label{uniq,f,g,v}
Let $f,g\in\mathcal{D}'(\mathbb{R})$. Then, the distribution $v$ given above in Theorem \ref{exs,f,g} is the unique solution in $\mathcal{C}^1(\mathbb{R},\mathcal{D}'(\mathbb{R}))$ of problem \eqref{main,v}.
\end{theorem}
These results allow us to solve also problem (\ref{main,u}) uniquely as stated in the following:
\begin{remark}\label{exs,uniq,f,g,u}
Let $f,g\in\mathcal{D}'(\mathbb{R})$. The distribution $u$ given by
\begin{equation*}
\begin{split}
\left\langle u,\varphi\right\rangle &:= e^{-\frac{k}{2}t}\left\lbrace\left \langle f,\left( \psi_t+\frac{k}{2}\psi\right) \ast\varphi \right\rangle +\left\langle g,\psi\ast\varphi \right\rangle \right\rbrace \\
&= e^{-\frac{k}{2}t}\left\lbrace \left\langle f,\psi_t\ast\varphi \right\rangle +\left\langle \frac{k}{2}f+ g,\psi\ast\varphi \right\rangle \right\rbrace
\end{split}
\end{equation*}
is the unique solution in $\mathcal{C}^1(\mathbb{R},\mathcal{D}'(\mathbb{R}))$ of problem \eqref{main,u}.
\end{remark}

The structure of the paper is as follows. In Section \ref{heuristics} we deduce heuristically a possible solution to problem \ref{main,v} for $f=0$ and $g=\delta$. In Section \ref{rf}, inspired by the heuristics, we prove rigorously Theorem \ref{exs,psi}. Then, in Section \ref{rfd} we prove Theorems \ref{exs,f,g} and \ref{uniq,f,g,v}.

Finally, in Section \ref{sg} we give some properties of the semigroup that generates the solution when the initial conditions $(f,g)\in H^1(\mathbb{R})\times L^2(\mathbb{R})$ and in Section \ref{fm} we apply Remark \ref{exs,uniq,f,g,u} to the price evolution model of a financial asset proposed in \cite{Kol}.

While many results are thoroughly proved here, for the sake of simplicity the complete calculations of some proofs have been ommitted but can be found in the author's Bachelor Degree Thesis \cite{Nua} which was advised by Joan Sol\`{a}-Morales.

\section{Heuristics}\label{heuristics}
Inspired by \cite{G&L}, we introduce the characteristic coordinates for the 1-D wave equation, which are $\zeta = ct-x$ and $\eta = ct + x$. If $\psi(x,t)$ solves $\psi_{tt}-c^2\psi_{xx}=\frac{k^2}{4}\psi$ then $v(\zeta,\eta)=\psi(x,t)$ solves  $v_{\zeta \eta}=\frac{k^2}{16c^2}v$. At this point \cite{Ock} uses Fourier Transforms, although it also mentions the possibility of proceeding in other ways.

We notice that for all $a\neq 0$, the function $w(\zeta,\eta)=v\left (a\zeta,\frac{1}{a}\eta \right )$ also solves $w_{\zeta \eta}=\frac{k^2}{16c^2}w$, therefore we have a family of solutions that depends on the parameter $a$, the equation is invariant under the transformation $(\zeta,\eta)\rightarrow \left( a\zeta, \frac{1}{a} \eta \right )$.
For this reasons, we look for solutions invariant under this kind of transformations, for example we look for solutions of the form $\psi(x,t)=v(\zeta,\eta)=f(\zeta\eta)=f(c^2t^2-x^2)=f(\lambda)=h(2\alpha\sqrt{\lambda})=h(\xi)$ for some $f$ and $h$ to be found, where we write $\lambda=c^2t^2-x^2$ and $\xi=2\alpha\sqrt{\lambda}$. This way,
\begin{equation*}
0=\psi_{tt}-c^2\psi_{xx}-\frac{k^2}4\psi
=4c^2\left[ \lambda f''+f'-\alpha^2f\right] 
=\frac{c^2}{\lambda}\left[\xi^2h''(\xi)+\xi h'(\xi)-\xi^2h(\xi)\right].
\end{equation*}
The solution on $h$ is a linear combination of the modified Bessel functions of order $n=0$,
\begin{equation*}
h(\xi)=AI_0(\xi)+BK_0(\xi)
\end{equation*}
when $\lambda\geq0$. Otherwise, we extend it by 0 and we write it by
\begin{equation*}
\psi(x,t)=\left[ AI_0(2\alpha\sqrt\lambda)+BK_0(2\alpha\sqrt\lambda)\right]\mathcal{X}_{[-c|t|, c|t|]}(x).
\end{equation*}
Let us now consider $g(x)$ a function at least continuous. If we want $v_t(x,0)=g(x)$, let us take
\begin{equation*}
v(x,t)=(g\ast\psi)(x)(t)=\int_{x-ct}^{x+ct}\left[AI_0(2\alpha\sqrt\lambda)+BK_0(2\alpha\sqrt\lambda)\right]g(y) dy,
\end{equation*}
with $A$ and $B$ still to be determined. We compute
\begin{equation*}
\begin{split}
v_t(x,t)&=c\left[AI_0(2\alpha\sqrt0)+BK_0(2\alpha\sqrt0)\right]\left[g(x+ct)+g(x-ct)\right] \\
&\quad+\int_{x-ct}^{x+ct}\left[AI_0'(2\alpha\sqrt\lambda)+BK_0'(2\alpha\sqrt\lambda)\right]\frac{2\alpha c^2t}{\sqrt\lambda}g(y) dy
\end{split}
\end{equation*}
since $\lambda=0$ when $y=x\pm ct$. We also have that $I_0(0)=1$ and $K_0(z)\rightarrow \infty$ when $z\rightarrow 0$ so for $v_t$ to exist and be bounded we impose $B=0$. Therefore, it reduces to
\begin{equation*}
v_t(x,t)=cA\left[g(x+ct)+g(x-ct)\right]+\int_{x-ct}^{x+ct}AI_0'(2\alpha\sqrt\lambda)\frac{2\alpha c^2t}{\sqrt\lambda}g(y) dy.
\end{equation*}
We have $\lim_{z\rightarrow0}\frac{I_0'(2\alpha z)}z=\alpha$. It is then clear that
\begin{equation*}
\lim_{t\rightarrow0}\int_{x-ct}^{x+ct}AI_0'(2\alpha\sqrt\lambda)\frac{2\alpha c^2t}{\sqrt\lambda}g(y) dy=0
\end{equation*}
because the interval of integration reduces only to the point $x$ and the integrand is not only bounded but also tends to 0 as $t$ does so. Hence, we have
$
\lim_{t\rightarrow0}v_t(x,t)=2cAg(x)=g(x)
$
assuming $A=\frac1{2c}$ and $g$ continuous. We have just seen $v(x,t)$ satisfies one initial condition. As for the other, let us use all we have deduced so far. From
\begin{equation*}
v(x,t)=\int_{x-ct}^{x+ct}\frac1{2c}I_0(2\alpha\sqrt\lambda)g(y) dy
\end{equation*}
and using the same argument, when $t\rightarrow0$ the integral reduces to the single point $x$ and the integrand is bounded since $I_0(z)\rightarrow1$ as $z\rightarrow0$ we can easily show that $v(x,0)=0$. 
All we have done is quite heuristic, not very rigorous. However, it has provided us a useful insight into the equation and its properties, as well as a good candidate for the real solution of the problem. Let us now go back and study the fundamental solution of the problem.
\section{Proof of Theorem \ref{exs,psi}}\label{rf}
We must prove that $\psi(x,t)$ belongs to ${\mathcal C}^1(\mathbb{R},\mathcal{D}'(\mathbb{R}))$ and as a two-variables distribution, $\psi(x,t)\in\mathcal{D}'(\mathbb{R}\times\mathbb{R})$ solves (\ref{main,v}) in the sense of distributions when $f=0$ and $g=\delta$ the Dirac delta centered at $x=0$. We first check that $\psi$ solves the differential equation and afterwards we prove that it belongs to ${\mathcal C}^1(\mathbb{R},\mathcal{D}'(\mathbb{R}))$, satisfies the required initial conditions, and we compute its time-derivative.
\begin{proof}[Proof that $\psi$ solves (\ref{main,v}).]
First of all, let us remark that when $x=\pm ct$ we have $\psi(x,t)\equiv sgn(t)\frac{1}{2c}$, notice there is a discontinuity in the straight lines $\left\lbrace x=ct\right\rbrace $ and $\left\lbrace x=-ct \right\rbrace $, the $\textit{characteristics}$ of our equation because of the discontinuities of $\mathcal{X}_{[-c|t|,c|t|]}(x)$. 
Hence, we can't expect our candidate to be a classical solution of the problem. However, let us see it is a solution in the sense of distributions

Let $L$ be the differential operator defined by
\begin{equation*}
L(u)=u_{tt}-c^2u_{xx}-\frac{k^2}4u,
\end{equation*}
the aim of the proof is to show that $\psi(x,t)$ is such that $\left\langle L(\psi),\varphi\right\rangle=0$, for all $\varphi\in{\mathcal C}_0^\infty\left(\mathbb{R}^2\right)$, that is,
\begin{equation}\label{def.sol}
\begin{split}
\left\langle L(\psi),\varphi\right\rangle=
\left\langle \psi,L^*(\varphi)\right\rangle&=\int_{\mathbb{R}^2}\psi L^*(\varphi)dx dt \\
&=\int_{\mathbb{R}^2}\psi\left( \varphi_{tt}-c^2\varphi_{xx}-\frac{k^2}4\varphi \right)dxdt=0,
\end{split}
\end{equation}
by definition of the adjoint of the operator, which in this case it is itself. Note that this last integral is well defined.
 
Let then $\varphi\in\mathcal{C}^{\infty}_0(\mathbb{R}^2)$ be a smooth function with compact support $K$, say included in the rectangle $R=[-a,a]\times[-\frac{a}{c},\frac{a}{c}]\subset\mathbb{R}\times\mathbb{R}$ for some $a>0$. Below there's a picture of the situation. The rectangle represents $R$ and the shaded area the region where $\psi$ is not 0 (See Figure \ref{fig1}).

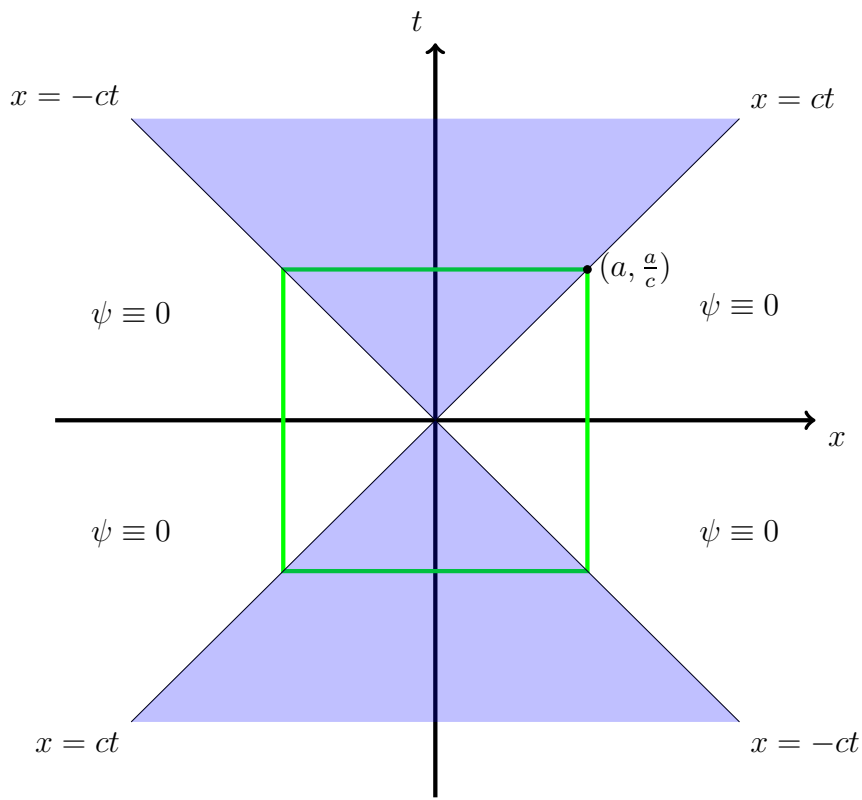
\begin{figure}[h!]
\begin{center}
\begin{tikzpicture}
\draw[ultra thick,->] (-5,0) -- (5,0) node[anchor=north west] {$x$ };
\draw[ultra thick,->] (0,-5) -- (0,5) node[anchor=south east] {$t$ };
\draw[ultra thick, green] (-2,-2) rectangle (2,2);
\draw (-4,-4) -- (4,4);
\draw (-4,4) -- (4,-4);
\fill[fill=blue, opacity=0.25] (0,0) -- (4,4) -- (-4,4) -- cycle;
\fill[fill=blue, opacity=0.25] (0,0) -- (4,-4) -- (-4,-4) -- cycle;
\node at (4,1.5) {$\psi\equiv 0$};
\node at (4,-1.5) {$\psi\equiv 0$};
\node at (-4,1.4) {$\psi\equiv 0$};
\node at (-4,-1.5) {$\psi\equiv 0$};
\node [above right] at (4,4) {$x=ct$};
\node [below left] at (-4,-4) {$x=ct$};
\node [above left] at (-4,4) {$x=-ct$};
\node [below right] at (4,-4) {$x=-ct$};
\draw[fill] (2,2) circle [radius=0.05];
\node [right] at (2,2) {$(a,\frac{a}{c})$};
\end{tikzpicture}
\end{center}
\caption{Scheme of the situation when integrating}
\label{fig1}
\end{figure}

Due to the compact support $K\subset R$ of $\varphi$ and the fact that $\psi$ is identically zero whenever $|x|>c|t|$, we have that the domain of integration of the last integral in \eqref{def.sol} can be reduced to two distinct open triangles inside the rectangle R, which are
\begin{equation*}
\begin{split}
\bigtriangledown&=\left\lbrace (x,t)\in\mathbb{R}^2 \mid 0<t<\frac{a}{c},\;-ct<x<ct\right\rbrace, \\
\bigtriangleup&=\left\lbrace (x,t)\in\mathbb{R}^2 \mid -\frac{a}{c}<t<0,\;ct<x<-ct\right\rbrace.
\end{split}
\end{equation*}
Then,
\begin{equation*}
\int_{\mathbb{R}^2}\psi\left(\varphi_{tt}-c^2\varphi_{xx}\right)=\int_\bigtriangledown \psi\text{div}X+\int_\bigtriangleup \psi\text{div}X,
\end{equation*}
where $X=\begin{pmatrix}-c^2\varphi_x\\\varphi_t\end{pmatrix}$. Let us reason for the integral on $\bigtriangledown$, the argument is the same for the other. The vector field $\psi X$ is continuously differentiable up to the boundary because both $\psi$ and $X$ are so $\textit{inside}$ the triangle. Hence, we can use the divergence theorem. 
\begin{equation*}
\int_\bigtriangledown \psi\text{div}X=\int_\bigtriangledown\text{div}(\psi X)-\int_\bigtriangledown\nabla \psi\cdot X=\int_{\vee}\psi X\cdot n\;d\ell-\int_\bigtriangledown\nabla \psi\cdot X.
\end{equation*}
Notice also that
\begin{equation*}
\begin{split}
\int_\bigtriangledown\nabla \psi\cdot X=\int_\bigtriangledown\begin{pmatrix}\psi_x&\psi_t\end{pmatrix}\cdot\begin{pmatrix}-c^2\varphi_x\\\varphi_t\end{pmatrix} 
=\int_\bigtriangledown Y\cdot\nabla\varphi=\int_\bigtriangledown\text{div}\left(\varphi Y\right)-\int_\bigtriangledown\varphi\text{div}Y,
\end{split}
\end{equation*}
where $Y=\begin{pmatrix}-c^2\psi_x\\\psi_t\end{pmatrix}$. The vector field $\varphi Y$ is continuously differentiable up to the boundary because both $\psi$ and $X$ are so $\textit{inside}$ the triangle. Hence, we can use the divergence theorem again to obtain
\begin{equation*}
\int_\bigtriangledown\nabla \psi\cdot X=\int_{\vee}\varphi Y\cdot n\;d\ell-\int_\bigtriangledown\varphi\text{div}Y.
\end{equation*}
This way we can write 
\begin{equation*}
\int_\bigtriangledown \psi\text{div}X=\int_{\vee}\psi X\cdot n\;d\ell-\int_{\vee}\varphi Y\cdot n\;d\ell+\int_\bigtriangledown\varphi\text{div}Y,
\end{equation*}
and the same expression holds for the integral on $\bigtriangleup$. The boundary terms are integrated on 
\begin{equation*}
\begin{split}
\vee&=\left\lbrace (x,t)\in\mathbb{R}^2 \mid 0<t<\frac{a}{c},\;x=\pm ct\right\rbrace \text{ and } \\
\wedge&=\left\lbrace (x,t)\in\mathbb{R}^2 \mid -\frac{a}{c}<t<0,\;x=\pm ct\right\rbrace.
\end{split}
\end{equation*}
The boundaries of the triangles also contain the lines $[-a,a]\times\left\lbrace \frac{a}{c}\right\rbrace $ and $[-a,a]\times\left\lbrace -\frac{a}{c}\right\rbrace$. However, the integrals on such lines are 0 because they fall outside the compact support of $\psi X$ (due to $\varphi$) and therefore we just integrate on $\vee$ and $\wedge$ taking $n$ the exterior normal unit vector on such sets. See Figure \ref{fig1} to clarify the situation.

We will compute the integrals on $\vee$, the computation on $\wedge$ is similar and can be found in \cite{Nua}. In order to compute $\int_{\vee}\psi X\cdot n\;d\ell$, we divide it according to the sign of $x$ and apply the corresponding values of $n$ in each case. First, we parametrize the line $\sigma(s)=\left(s,\frac {-s}c\right)$ with $s\in[-a,0]$. Then,
\begin{multline*}
\int_{-a}^0\left[ \psi\begin{pmatrix}-c^2\varphi_x\\\varphi_t\end{pmatrix}\cdot\begin{pmatrix}-\frac1c&-1\end{pmatrix}\right] (\sigma(s))ds 
=c\int_{-a}^0\left[ \psi\left(\varphi_x-\frac1c\varphi_t\right)\right] (\sigma(s))ds  \\
=c\int_{-a}^0\left[\psi\frac{d\varphi}{ds}\right](\sigma(s))ds 
=c\left(\left(\psi\varphi\right)(0,0^+)-\int_{-a}^0\left[\varphi\frac{d\psi}{ds}\right]\left(s,\frac {-s}c\right)ds\right).
\end{multline*}
We now parametrize the other part of the set, the line given by $\sigma(s)=\left(s,\frac sc\right)$ with $s\in[0,a]$. Using the same parts integration and directional derivative strategy as before,
\begin{multline*}
\int_0^a\left[ \psi\begin{pmatrix}-c^2\varphi_x\\\varphi_t\end{pmatrix}\cdot\begin{pmatrix}\frac1c&-1\end{pmatrix}\right] \left(\sigma(s)\right)ds 
=c\left(\left(\psi\varphi\right)(0,0^+)+\int_0^a\left[\varphi\frac{d\psi}{ds}\right]\left(s,\frac sc\right)ds\right).
\end{multline*}
Consequently, we have that
\begin{multline*}
\int_{\vee}\psi X\cdot n\;d\ell
=2c(\psi\varphi)(0,0^+) 
+c\int_0^a\left[\varphi\frac{d\psi}{ds}\right]\left(s,\frac sc\right)ds  
-c\int_{-a}^0\left[\varphi\frac{d\psi}{ds}\right]\left(s,\frac {-s}c\right)ds
\end{multline*}
For $\wedge$ a similar calculation yields
\begin{multline*}
\int_{\wedge}\psi X\cdot n\;d\ell
=2c(\psi\varphi	)(0,0^-) 
+c\int_{0}^{a}\left[\varphi\frac{d\psi}{ds}\right]\left(s,-\frac {s}c\right)ds-c\int_{-a}^{0}\left[\varphi\frac{d\psi}{ds}\right]\left(s,\frac {s}c\right)ds.
\end{multline*}
On the other hand, to compute $\int_{\vee}\varphi Y\cdot n\;d\ell$ we parametrize again one subset by $\sigma(s)=\left(s,\frac {-s}c\right)$ with $s\in[-a,0]$. Then,
\begin{multline*}
\int_{-a}^0\left[ \varphi\begin{pmatrix}-c^2\psi_x\\\psi_t\end{pmatrix}\cdot\begin{pmatrix}-\frac1c&-1\end{pmatrix}\right] \left(\sigma(s)\right)ds  
=\int_{-a}^0\left[ \varphi\left(c\psi_x-\psi_t\right)\right] \left(\sigma(s)\right)ds  \\
=c\int_{-a}^{0}\left[ \varphi\left(\psi_x-\frac1c\psi_t\right)\right] \left(\sigma(s)\right)ds 
=c\int_{-a}^0\left[\varphi\frac{d\psi}{ds}\right]\left(s,-\frac sc\right)ds.
\end{multline*}
The other subset is parametrized by $\sigma(s)=\left(s,\frac sc\right)$ with $s\in[0,a]$. Using the same strategy as above, we have
\begin{equation*}
\int_0^a\left[\varphi\begin{pmatrix}-c^2\psi_x\\\psi_t\end{pmatrix}\cdot\begin{pmatrix}\frac1c&-1\end{pmatrix}\right] \left(\sigma(s)\right)ds=-c\int_0^a\left[\varphi\frac{d\psi}{ds}\right]\left(s,\frac sc\right)ds.
\end{equation*}
Therefore,
\begin{equation*}
\int_{\vee}\varphi Y\cdot n\;d\ell=c\int_{-a}^0\left[\varphi\frac{d\psi}{ds}\right]\left(s,-\frac sc\right)ds-c\int_0^a\left[\varphi\frac{d\psi}{ds}\right]\left(s,\frac sc\right)ds
\end{equation*}
and for $\wedge$	 a similar calculation yields
\begin{equation*}
\int_{\wedge}\varphi Y\cdot n\;d\ell=c\int_{-a}^{0}\left[\varphi\frac{d\psi}{ds}\right]\left(s,\frac sc\right)ds-c\int_{0}^{a}\left[\varphi\frac{d\psi}{ds}\right]\left(s,-\frac sc\right)ds.
\end{equation*}
Finally, mixing all things found so far and adding the missing term,
\begin{multline*}
\int_{\mathbb{R}^2}\psi\left(\varphi_{tt}-c^2\varphi_{xx}-\frac{k^2}4\varphi\right) dx dt =2c(\psi\varphi)(0,0^+)+2c(\psi\varphi)(0,0^-) \\
+2c\int_0^a\left[\varphi\frac{d\psi}{ds}\right]\left(s,\frac sc\right)ds -2c\int_{-a}^0\left[\varphi\frac{d\psi}{ds}\right]\left(s,\frac {-s}c\right)ds \\ -2c\int_{-a}^{0}\left[\varphi\frac{d\psi}{ds}\right]\left(s,\frac {s}c\right)ds +2c\int_{0}^{a}\left[\varphi\frac{d\psi}{ds}\right]\left(s,-\frac {s}c\right)ds \\
+\int_{\lbrace\bigtriangledown\cup\bigtriangleup\rbrace}\varphi\left(\psi_{tt}-c^2\psi_{xx}-\frac{k^2}4\psi\right) dx dt.
\end{multline*}
The notation $0^+$ means the limit to zero going from above and $0^-$ is the limit to zero going from below. This distinction is crucial since we have $\psi$ defined differently whereas $t>0$ or not. In our case, we have $\psi(0,0^+)\rightarrow\frac{1}{2c}$ and $\psi(0,0^-)\rightarrow-\frac{1}{2c}$. Since $\varphi\in \mathcal{C}^\infty_0(\mathbb{R}^2)$, we have  $2c(\psi\varphi)(0,0^+)+2c(\psi\varphi)(0,0^-)=0$.

Let us recall now a property of our candidate to solution, which is $\psi$ is either $\frac{1}{2c}$ or $-\frac{1}{2c}$ on the lines $x=\pm ct$. That means $\psi$ is constant on such lines and therefore its directional derivative is 0. We are just saying that
\begin{equation*}
\frac{d\psi}{ds}\left(s,\frac sc\right)=\frac{d\psi}{ds}\left(s,-\frac sc\right)=0,
\end{equation*}
for all $s\in\mathbb{R}$. Hence, the line integrals are all 0 and what remains to be seen is that
\begin{equation*}
\int_{\lbrace\bigtriangledown\cup\bigtriangleup\rbrace}\varphi\left(\psi_{tt}-c^2\psi_{xx}-\frac{k^2}4\psi\right) dx dt=0,
\end{equation*}
which is true provided that $\psi_{tt}-c^2\psi_{xx}-\frac{k^2}4\psi=0$ in the interior of the two triangles. We just check the result for positive times. For negatives times, the minus sign does not affect at all the result.
Let us once again recall what form has our candidate of solution inside the integrating region for positive times. It is
\begin{equation*}
\psi(x,t)=\frac1{2c}I_0\left(2\alpha\sqrt{c^2t^2-x^2}\right).
\end{equation*}
This function is infinitely differentiable in both variables in this region and for comfort let us denote $\lambda=c^2t^2-x^2$. To verify it satisfies the PDE, we use the same reasoning as in the Heuristics Section \ref{heuristics}. That is, we write $\psi(x,t)=\frac{1}{2c}I_0(2\alpha\sqrt{\lambda})=f(\lambda)=g(2\alpha\sqrt{\lambda})=g(\xi)$ for some $f$ and $g$ to determine.
Hence, recalling $\alpha=\frac{k}{4c}$,and using the relations between $f$ and $g$ we deduce that
\begin{equation*}
\psi_{tt}-c^2\psi_{xx}-\frac{k^2}4\psi
=4c^2\left[ \lambda f''+f'-\alpha^2f\right] 
=\frac{c^2}{\lambda}\left[\xi^2g''(\xi)+\xi g'(\xi)-\xi^2g(\xi)\right].
\end{equation*}
Finally, we have to check that $0=\frac{c^2}{\lambda}\left[\xi^2g''(\xi)+\xi g'(\xi)-\xi^2g(\xi)\right]$. Fortunately, it is trivial since $g(\xi)=\frac{1}{2c}I_0(\xi)$ is a multiple of the modified Bessel equation of order 0 and parameter 1, which is precisely defined as the function that solves this ODE and this finishes the proof.
\end{proof}
We now prove the other part of Theorem \ref{exs,psi}, which refers to $\psi$ belonging to ${\mathcal C}^1(\mathbb{R},\mathcal{D}'(\mathbb{R}))$ and the initial conditions. This result was inspired by \cite{Tay}.
\begin{proof}[Proof of $\psi(x,t)\in{\mathcal C}^1(\mathbb{R},\mathcal{D}'(\mathbb{R}))$]
We have to see that there exists
$
\lim_{h\rightarrow0}\frac{\psi(\cdot,t+h)-\psi(\cdot,t)}h=\psi_t(\cdot,t)\in\mathcal C^0\left(\mathbb{R},\mathcal D'(\mathbb{R})\right)
$
with the limits computed using the topology of the arrival space $\mathcal D'(\mathbb{R})$, that is
\begin{equation*}
\left\langle\frac{\psi(\cdot,t+h)-\psi(\cdot,t)}h,\varphi\right\rangle\xrightarrow{h\rightarrow0}\left\langle \psi_t(\cdot,t),\varphi\right\rangle,
\end{equation*}
for all $\varphi\in\mathcal D(\mathbb{R})$ and $t\in\mathbb{R}$, and that this limit is continuous, that is
\begin{equation*}
\lim_{h\rightarrow0}\left\langle \psi_t(\cdot,t+h),\varphi\right\rangle =\left\langle \psi_t(\cdot,t),\varphi\right\rangle,
\end{equation*}
for all $\varphi\in\mathcal D(\mathbb{R})$ and $t\in\mathbb{R}$, where the notation $\left\langle\psi(\cdot,t),\varphi\right\rangle$ refers here to $\int_{\mathbb{R}}\psi(x,t)\varphi(x)dx$. Here we just prove the case $t=0$, the general case is done similarly and can be found in \cite{Nua}. Notice $\psi(x,0)=0$ is expected from the odd symmetry of the function. Indeed, we will see that with this choice, we will have a differentiable application, in particular a continuous one. 

Let us first take $1\gg h>0$ and $\varphi\in{\mathcal C}_0^\infty(\mathbb{R})$,
\begin{equation*}
\begin{split}
\left\langle \frac{\psi(\cdot,h)}{h},\varphi \right\rangle&=\int_{\mathbb{R}}\frac{1}{h}\psi(x,h)\varphi(x)ds 
\overset{h>0}=
\int_{\mathbb{R}}\frac{1}{2c}\frac{I_0\left(2\alpha \sqrt{c^2h^2-x^2} \right) }{h}\mathcal{X}_{[-ch,ch]}(x)\varphi(x)dx \\
&= \frac{1}{h}\int_{-ch}^{ch}\frac{1}{2c}I_0\left(2\alpha \sqrt{c^2h^2-x^2} \right)\varphi(x)dx.
\end{split}
\end{equation*}
Then, when $h\rightarrow0$, the interval of the integral reduces to 0 while the integrand tends to $\frac{1}{2c}$, it is bounded. So, the integral goes to 0 and then the quotient tends to $\frac{0}{0}$. Using Hopital's Rule,
\begin{multline*}
\frac d{dh}\left(\int_{-ch}^{ch}\frac1{2c}I_0\left(2\alpha\sqrt{c^2h^2-x^2}\right)\varphi(x)dx\right) \\
=\frac1{2c}\left\lbrace I_0(0)\varphi(ch)c -I_0(0)\varphi(-ch)(-c)\right\rbrace 
+\int_{-ch}^{ch}\frac1{2c}\frac{I_0'\left(2\alpha\sqrt{c^2h^2-x^2}\right)}{2\sqrt{c^2h^2-x^2}}2\alpha2c^2h\varphi(x)dx
\\
=\frac1{2}\left\lbrace \varphi(ch)+\varphi(-ch)\right\rbrace
+\int_{-ch}^{ch}\alpha ch\frac{I_0'\left(2\alpha\sqrt{c^2h^2-x^2}\right)}{2\sqrt{c^2h^2-x^2}}\varphi(x)dx,
\end{multline*}
after applying $I_0(0)=1$ and simplifying terms.  Letting $h\rightarrow0$ the integral vanishes since the integrand goes to 0 and the interval collapses and so we are left with
\begin{equation*}
=\frac1{2}\left\lbrace \varphi(0)+\varphi(0)\right\rbrace =\varphi	(0)=\left\langle \delta,\varphi \right\rangle.
\end{equation*}
When we do the limit with a negative $h$ we obtain the same result, concluding that $\psi_t(\cdot,0)=\delta(\cdot)\in\mathcal{D}'(\mathbb{R})$.

To see the continuity, we want to see that
\begin{equation*}
\lim_{h\rightarrow0}\left\langle \psi_t(\cdot,h),\varphi\right\rangle =\left\langle \psi_t(\cdot,0),\varphi\right\rangle=\left\langle \delta,\varphi\right\rangle.
\end{equation*}
Here we just present the case where $h\rightarrow 0^+$, the other one is similar. We have
\begin{equation*}
\left\langle \psi_t(\cdot,h),\varphi\right\rangle =\frac{1}{2}\varphi(ch)+\frac{1}{2}\varphi(-ch) 
+\int_{-ch}^{ch}\alpha ch\frac{I_0'\left(2\alpha\sqrt{c^2h^2-x^2}\right)}{\sqrt{c^2h^2-x^2}}\varphi(x)dx.
\end{equation*}
which is precisely what we obtained a few lines above, we conclude that $\lim_{h\rightarrow0}\left\langle \psi_t(\cdot,h),\varphi\right\rangle =\left\langle \psi_t(\cdot,0),\varphi\right\rangle=\left\langle \delta,\varphi\right\rangle$.
\end{proof}
\section{The General Initial Value Problem}\label{rfd}
Here we will prove Theorem \ref{exs,f,g} and Theorem \ref{uniq,f,g,v}. To do so, we will use the next intermediate results.
\begin{proposition}\label{rfd,ex,v,g}
Let $\psi(x,t)=\text{sgn(t)}\frac1{2c}I_0\left(2\alpha\sqrt{c^2t^2-x^2}\right){\mathcal X}_{\lbrack-c|t|,c|t|\rbrack}(x)$ and $g\in\mathcal{D}'(\mathbb{R})$ a distribution. Then, the distribution $v$ defined by
\begin{equation*}
\left\langle v(t),\varphi\right\rangle:=\left\langle g,\psi(t)\ast\varphi\right\rangle, 
\end{equation*}
for all $\varphi \in\mathcal{C}_0^\infty(\mathbb{R})$ is such that $v\in\mathcal C^1\left(\mathbb{R},\mathcal D'\left(\mathbb{R}\right)\right)$, as a two variable distribution is a solution in the distributional sense of
\begin{equation*}
\begin{array}{l}\left\{\begin{array}{cc}v_{tt}=c^2v_{xx}+\frac{k}{4}v\\v(0)=0, \quad v_t(0)=g&\end{array}\right.\\\end{array}
\end{equation*}
and its time-derivative is the distribution defined by 
\begin{equation*}
\left\langle v_t(t),\varphi\right\rangle :=\left\langle g,\psi_t(t)\ast\varphi\right\rangle 
\end{equation*}
for all $\varphi \in\mathcal{C}_0^\infty(\mathbb{R})$.
\end{proposition}
Before the proof, let us state these rather useful Remarks.
\begin{remark}
This candidate to solution $v$ may be understood as a resulting distribution on the $x$ variable for each $t\in\mathbb{R}$ or also as a two dimensional distribution. We will use this double meaning in different parts of the proof.
\end{remark}
\begin{remark}\label{remimp}
From Theorem \ref{exs,psi}, we know that
\begin{equation*}
\left\langle L(\psi),\phi\right\rangle=\int_{\mathbb{R}^2}\psi L^*(\phi) dxdt=0, 
\end{equation*}
for all $\phi\in\mathcal{C}_0^\infty(\mathbb{R}^2)$, a property that will be essential in the proof.
\end{remark}
\begin{remark}\label{rem,g,l}
In certain cases, we \emph{may} write
\begin{equation*}
v(x,t)=\int_{\mathbb{R}}\psi(x-y,t)g(y)dy 
= \int_{x-ct}^{x+ct}\frac1{2c}I_0(2\alpha\sqrt{c^2t^2-(x-y)^2})g(y) dy,
\end{equation*}
when $g$ is such that this expression makes sense, for example when $g\in L^2$.
\end{remark}
\begin{proof}[Proof of Proposition \ref{rfd,ex,v,g}]
First of all, let us see that our solution is well-defined, observe that $\psi(t)$ is a function with compact support and continuous inside $[-c|t|,c|t|]$. Therefore, $\psi(t) \ast \varphi\in\mathcal{C}_0^\infty(\mathbb{R})$ and then it makes sense to compute the action of the distribution $g$ on $\psi(t) \ast \varphi$, so $v$ is well-defined, it is a distribution.

Let us now interpret our candidate to solution as a two dimensional distribution. For $L$ the same differential operator as before, we have to give some sense to $\left\langle L(v),\phi\right\rangle=0$, a distribution acting on a two-variable test functions.  A reasonable definition would be 
\begin{equation*}
\left\langle L(v),\phi\right\rangle:=
\int_\mathbb{R} \left\langle v(t),L^*(\phi)\right\rangle dt
\end{equation*}
because it is fully well-defined and coincides with the action of any $h\in L^1_{loc}(\mathbb{R}^2)$ against a two-variable test function. We proceed
\begin{equation*}
\begin{split}
\int_\mathbb{R} \left\langle v(t),L^*(\phi)\right\rangle dt=\int_\mathbb{R} \left\langle g,\psi(t)\ast L^*(\phi(t))\right\rangle dt 
= \left\langle g,\int_\mathbb{R}\psi(t)\ast L^*(\phi(t))dt\right\rangle.
\end{split}
\end{equation*}
Notice we can enter the integral inside $\left\langle \cdot,\cdot\right\rangle $ because the duality product is indeed a linear continuous form and the integral is a limiting process based on sums that only concerns $\psi(t)\ast L`^*(\phi (t))$. Now, we just study
\begin{equation*}
\int_\mathbb{R}\psi(t)\ast L^*(\phi(t))dt= \int_{\mathbb{R}^2} \psi(y,t) L^*(\phi(x-y,t))dydt.
\end{equation*}
For a fixed $x\in\mathbb{R}$, if we write $\phi(x-y,t)=\varphi(y,t)$, we have that $\varphi\in\mathcal{C}_0^\infty(\mathbb{R}^2)$ and $L^*(\phi(x-y,t))=L^*(\varphi(y,t))\in\mathcal{C}_0^\infty(\mathbb{R}^2)$ so that
\begin{equation*}
\left\langle g,\underbrace{\int_{\mathbb{R}^2}\psi(y,t)L^\ast(\varphi(y,t))dydt}_0\right\rangle =0,
\end{equation*}
thanks to what we commented in Remark \ref{remimp}. This way, we make sure that the distribution $v$ is a solution in the sense of distributions of the differential equation.

The first initial condition should be $v(x,0)=0$ in the sense of distributions in $x$. Indeed, $\left\langle v(0),\varphi\right\rangle:=\left\langle g,\psi(0)\ast\varphi\right\rangle=\left\langle g,0\ast\varphi\right\rangle=0$, for all $\varphi\in\mathcal{C}_0^\infty(\mathbb{R})$, so that $v(0)=0$ as a distribution. Again, we will see that $v_t(0)=g$ as a consequence of $v(t)\in\mathcal{C}^1(\mathbb{R},\mathcal{D'(\mathbb{R})})$.
Thanks to Theorem \ref{exs,psi}, we know that $\psi(t)\in\mathcal{C}^1(\mathbb{R},\mathcal{D'(\mathbb{R})})$ and hence we can write
\begin{equation*}
\left\langle v_t(t),\varphi\right\rangle :=\left\langle g,\psi_t(t)\ast\varphi\right\rangle.
\end{equation*}
This expression is well defined since $\psi_t$ is compactly supported both in the distributional and functional sense and therefore $\psi_t\ast\varphi\in\mathcal{C}_0^\infty(\mathbb{R})$. What is more, since $\psi(t)\in\mathcal{C}^1(\mathbb{R},\mathcal{D'(\mathbb{R})})$ then $v(t)\in\mathcal{C}^1(\mathbb{R},\mathcal{D'(\mathbb{R})})$.
As a result, for the second initial condition, for $t=0$ we have $\left\langle v_t(0),\varphi\right\rangle=\left\langle g,\psi_t(0)\ast\varphi\right\rangle=\left\langle g,\delta\ast\varphi\right\rangle=\left\langle g,\varphi\right\rangle$ so we conclude that $v_t(0)=g$ as a distribution.
\end{proof}
\begin{remark}\label{rem,trans}
According to what we have seen during the proof, any distribution $v$ defined by $\left\langle v(t),\varphi\right\rangle:=\left\langle g,\psi(t)\ast\varphi\right\rangle$ satisfies any PDE in the sense of distributions if its \emph{kernel} or fundamental solution $\psi(t)$ also does so.
\end{remark}
\begin{lemma}\label{cpr,lema}
Let $g\in\mathcal{D}'(\mathbb{R})$. Then, the distribution $w$ defined by
\begin{equation*}
\left\langle w,\varphi\right\rangle :=\left\langle g,\psi_t\ast\varphi\right\rangle
\end{equation*}
for all $\varphi\in\mathcal{C}_0^\infty(\mathbb{R})$ solves
\begin{equation*}
\begin{array}{l}\left\{\begin{array}{cc}w_{tt}=c^2w_{xx}+\frac{k}{4}w\\w(0)=g, \quad w_t(0)=0&\end{array}\right.\\\end{array}
\end{equation*}
in the sense of distributions.
\end{lemma}
\begin{proof}
Firstly, it is reasonable to define $w$ this way. Notice $\psi_t$ is compactly supported as a distribution and, as a result, $\psi_t\ast\varphi\in\mathcal{C}_0^\infty(\mathbb{R})$. Let us now recall Remark \ref{rem,trans} which says that the distribution $w$ satisfies the PDE if $\psi_t(t)$ does so.
We know $\psi(t)$ does solve the PDE and if we derive the PDE of (\ref{main,v}) with respect to time we obtain that $\psi_t$ also solves it. Consequently, $w$ satisfies the PDE in the sense of distributions. As for the initial conditions, they are easily deduced (remember that $\psi_t(0)=\delta$).
\end{proof}
With these partial results we are now able to give a proof of Theorem \ref{exs,f,g}.
\begin{proof}[Proof of Theorem \ref{exs,f,g}]
We divide the problem into two smaller ones, one with homogeneous first initial condition and the other with homogeneous second initial condition. We use Lemma \ref{cpr,lema} for the problem that has a distribution as its first initial condition and Proposition \ref{rfd,ex,v,g}. Then, the solution for the general problem is the sum of these two partial solutions acting on the same test function.
\end{proof}
In order to prepare the proof of uniqueness, we present the following
\begin{proposition}\label{rfr,prop,s,g}
Let $\varphi\in\mathcal{C}_0^\infty(\mathbb{R})$, let $f,g\in\mathcal{D}'(\mathbb{R})$ and $v$ a solution in the sense of distributions of problem (\ref{main,v}).

The function $s(x)=v\ast\varphi(x):=\left\langle v,\varphi(x-\cdot) \right\rangle\in\mathcal{C}^\infty(\mathbb{R})$ is well defined and it is a classical solution in the sense of distributions of
\begin{equation}\label{classicproblem}
\begin{array}{l}\left\{\begin{array}{cc}s_{tt}=c^2s_{xx}+\frac{k}{4}s\\s(x,0)=(f\ast\varphi)(x), \quad s_t(x,0)=((g+\frac{k}{2}f)\ast\varphi)(x)&\end{array}\right.\\\end{array}
\end{equation}
\end{proposition}
\begin{proof}
Both initial conditions are satisfied easily thanks to the initial conditions that $v$ satisfies. As for the differential equation, we need
\begin{equation*}
\int_{\mathbb{R}^2}L(s)\phi\;dx dt:=\int_{\mathbb{R}^2}sL^*(\phi)dx dt=0,
\end{equation*}
for all $\phi\in\mathcal{C}_0^\infty(\mathbb{R}^2)$.
Let us remark that this time the integration is already well defined, since $s\in\mathcal{C}^\infty(\mathbb{R})$ and not a general distribution.
Therefore, let $\phi\in\mathcal{C}_0^\infty(\mathbb{R}^2)$, we can write
\begin{equation*}
\begin{split}
\int_{\mathbb{R}^2}sL^*(\phi)dx dt&=\int_\mathbb{R}\left( \int_{\mathbb{R}}s L^*(\phi)dx\right) dt \\
&=\int_\mathbb{R}\left( \int_{\mathbb{R}}\left\langle v,\varphi(x-\cdot) \right\rangle L^*(\phi)dx\right) dt \\
&=\int_\mathbb{R}\left( \int_{\mathbb{R}}\left\langle v,L^*(\phi)\varphi(x-\cdot) \right\rangle dt\right) dx 
\\
&=\int_\mathbb{R}\left( \int_{\mathbb{R}}\left\langle v,L^*(\phi(x,t))\varphi(x-\cdot) \right\rangle dt\right) dx \\
&\overset{x-\cdot=z}{\underset{dx=dz}=}\int_\mathbb{R}\left( \int_{\mathbb{R}}\left\langle v,L^*(\phi(z+\cdot,t))\varphi(z) \right\rangle dt\right) dz 
\\
&=\int_\mathbb{R}\varphi(z)\left( \int_{\mathbb{R}}\left\langle v,L^*(\phi(z+\cdot,t)) \right\rangle dt\right) dz.
\end{split}
\end{equation*}

Fixed $z\in\mathbb{R}$, if we write $\phi(z+y,t)=\gamma(y,t)$, we have that $\gamma\in\mathcal{C}_0^\infty(\mathbb{R}^2)$ and $L^*(\phi(z+y,t))=L^*(\gamma(y,t))\in\mathcal{C}_0^\infty(\mathbb{R}^2)$ so that
\begin{equation*}
\int_{\mathbb{R}}\varphi(z)\underbrace{\left(\int_{\mathbb{R}}\left\langle v(t),L^\ast(\gamma)\right\rangle dt\right)}_0dz=0
\end{equation*}
because $v$ is a solution in the distributional sense of problem (\ref{main,v}) and so $s$ solves the problem in the classical sense.
\end{proof}
Now we present the proof of the uniqueness based on the use of convolutions and classical solutions.
\begin{proof}[Proof of Theorem \ref{uniq,f,g,v}]
Assume $v$, $w$ are both solutions of (\ref{main,v}) and let $\varphi\in\mathcal{D}(\mathbb{R})$. We construct $s(x)=v\ast\varphi(x)$ and $r(x)=w\ast\varphi(x)$, both being $\mathcal{C}^\infty(\mathbb{R})$, by Proposition \ref{rfr,prop,s,g} they are classical solutions of problem (\ref{classicproblem}). If we see that $u(x)=s(x)-r(x)$ is the zero function then we are done.

First of all, let us remark that $u(x,t)$ is continuous in $\mathbb{R}^2$. Continuity on the space variable $x$ is granted by the convolution properties. Continuity on the time variable $t$ is deduced thanks to Theorem \ref{exs,f,g} and the definition of continuity in the space of distributions. This implies that $u$ will be bounded in any compact set.

Let us now take $T=\frac{2}{k}$ and $x\in\mathbb{R}$, we consider the characteristic triangle
\begin{equation*}
\mathcal T=\left\{(y,s)\in\mathbb{R}^2:\;s\in\lbrack0,T\rbrack,y\in\lbrack x-c(T-s),x+c(T-s)\rbrack\right\}.
\end{equation*}
We know that $u$ is continuous and this triangle is compact so there exists $M>0$ such that $|u(y,s)|\leq M$ there.
Since $u$ solves $u_{tt}=c^2u_{xx}+\frac{k^2}{4}u$ with homogeneous initial conditions ($f=g=0$), using D'Alembert formula we necessarily have that
\begin{equation*}
u(x,t)=\frac{1}{2c}\int_0^t \int_{x-c(t-s)}^{x+c(t-s)}\frac	{k^2}{4}u(y,s) dy ds.
\end{equation*}
Now, the pair $(y,s)$ is in our triangle $\mathcal{T}$ and so $|u(y,s)|\leq M$. Then,
\begin{equation*}
\begin{split}
|u(x,t)|\leq \frac{1}{2c}\int_0^t \int_{x-c(t-s)}^{x+c(t-s)}\frac	{k^2}{4}M dy ds\leq\frac{1}{2c} \frac{k^2}{4}M\int_0^T \int_{x-c(T-s)}^{x+c(T-s)}dy ds 
&=\frac{1}{2c}\frac{k^2}{4}cT^2M \\
&=\frac{M}{2}.
\end{split}
\end{equation*}
This way we see that if $|u(x,t)|\leq M$ then $|u(x,t)|\leq M/2$ from which we conclude that $u(x,t)=0$, for all $0 \leq t \leq T$. Notice this argument can be used at any $x\in\mathbb{R}$ because despite $M$ may depend on the point, the time $T$ for which we get this contraction does not. Hence, $u(x,t)=0$, for all $(x,t)\in \mathbb{R}\times[0,T]$.

Now, if we consider $\widetilde u(x,t)=u(x,T+t)$ whose initial conditions are $\widetilde u(x,0)=u(x,T)=0$ and $\widetilde u_t(x,0)=u_t(x,T)=0$ we can repeat the same argument and get that $\widetilde u(x,t)=0$, for all $(x,t)\in \mathbb{R}\times[0,T]$ and so $u(x,t)=0$, for all $(x,t)\in \mathbb{R}\times[0,2T]$. By induction, it follows easily that $u(x,t)=0$, for all $(x,t)\in \mathbb{R}\times\mathbb{R}^+$. The same can be done with negative times, $u(x,t)=0$ in $\mathbb{R}^2$ and we are done.
\end{proof}
\begin{remark}\label{Eu sol total dist}
The formula given by (\ref{exs,uniq,f,g,u})
is the unique solution in $\mathcal{C}^1(\mathbb{R},\mathcal{D}'(\mathbb{R}))$ of
problem (\ref{main,u}). It easily follows from taking $v=e^{\frac{k}{2}t}u$ and applying the previous results.
\end{remark}

\section{Application 1: The Semigroup}\label{sg}
We now study the problem (\ref{main,u})
when $(f,g)$ belong to the Hilbert Space $ H^1(\mathbb{R})\times L^2(\mathbb{R})$. We will apply the results found above for general distributions to obtain properties of the solution in this special case.
\begin{remark}\label{sg,th,u,f,g,h,l}
Given $(f,g)\in H^1(\mathbb{R})\times L^2(\mathbb{R})$, the solution to the initial value problem (\ref{main,u}) is given by
\begin{equation*}
\begin{split}
u(x,t)&=e^{\frac{-k}{2}t}\left\lbrace \frac{1}{2}\left[f(x+ct)+f(x-ct)\right]+\alpha ct\int_{x-ct}^{x+ct}\frac{I_0'(2\alpha\sqrt\lambda)}{\sqrt\lambda}f(y)dy \right. \\
&\qquad\qquad\left.+\int_{x-ct}^{x+ct}\frac1{2c}I_0(2\alpha\sqrt\lambda)\left[ \frac{k}{2}f(y)+g(y)\right] dy \right\rbrace,
\end{split}
\end{equation*}
with $\lambda=c^2t^2-(x-y)^2$.
\end{remark}
\begin{definition}
A family $\left\lbrace \Gamma_t \right\rbrace _{t\geq0}$ of bounded linear operators on a Banach space $X$ into itself is said to have the semigroup property if
\begin{equation*}
\Gamma_0=I_d \text{ and }\Gamma_t(\Gamma_s)=\Gamma_{t+s}\quad\forall t,s\geq0.
\end{equation*}
\end{definition}
Let us consider $X=H^1(\mathbb{R})\times L^2(\mathbb{R})$ and the operator that for each $t\geq0$ and for each $(f,g)\in X$ it gives us $\Gamma_t(f,g)=$ the solution and its time derivative of problem (\ref{main,u}),  with $f,g$ as the initial conditions. To ease the notations, we write $u(t)$ and $u_t(t)$ to refer to such solutions. We have the following
\begin{theorem}\label{sg,th}
Let $(f,g)\in X$. Then,
\begin{enumerate}
\item $\Gamma_t(f,g)=(u(t),u_t(t))\in X$.
\item $\Gamma_t$ has the semigroup property.
\item $\forall t\in\mathbb{R} \quad \exists M_1,M_2,M_3,N_1,N_2,N_3,P_1,P_2>0$ such that
\begin{enumerate}
\item $\Vert u \Vert_{L^2(\mathbb{R})} \leq e^{\frac{-k}{2}t}\left\lbrace M_1\Vert f \Vert_{L^2(\mathbb{R})} + N_1 \Vert g \Vert_{L^2(\mathbb{R})}\right\rbrace$,
\item $\Vert u_t \Vert_{L^2(\mathbb{R})} \leq e^{\frac{-k}{2}t}\left\lbrace M_2\Vert f \Vert_{L^2(\mathbb{R})} + N_2 \Vert g \Vert_{L^2(\mathbb{R})} + P_1\Vert f' \Vert_{L^2(\mathbb{R})}\right\rbrace$,
\item $\Vert u_x \Vert_{L^2(\mathbb{R})} \leq e^{\frac{-k}{2}t}\left\lbrace M_3\Vert f \Vert_{L^2(\mathbb{R})} + N_3 \Vert g \Vert_{L^2(\mathbb{R})} + P_2\Vert f' \Vert_{L^2(\mathbb{R})}\right\rbrace$,
\end{enumerate}
the semigroup acts continously.
\end{enumerate}
\end{theorem}
\begin{remark}
The proof of Theorem (\ref{sg,th}) can be found in \cite{Nua}, it relies on the fact that $u(x,t)$ is a combination of convolutions, the $L^2$ estimates of $u$ and its derivatives are obtained applying Young's Inequality properly on each convolution.
\end{remark}
\section{Application 2: A Financial Model}\label{fm}
Let us assume a particle moves on the discrete set $\left\lbrace k\Delta x\mid k\in\mathbb{Z}\right\rbrace \subset \mathbb{R}$ at time intervals of $\Delta t$.  Let us suppose that with probability $p$ it repeats the same move as in the previous jump and with probability $1-p$ does the contrary move. In the limiting process, when $\Delta x$ and $\Delta t$ are small, we will assume that $p$ is near 1, representing this way some kind of inertia in the movement.

In order to deduce which laws do the movement follow, let us denote the pair  $(k,n)=$the particle is in the position $x=k\Delta x$ at time $t=n\Delta t$. Let us also define
\begin{equation*}
\begin{split}
\alpha(k,n)=\text{P}\left( \text{the particle is in }(k,n) \text{ and comes from }(k-1,n-1)\right), \\
\beta(k,n)=\text{P}\left( \text{the particle is in }(k,n) \text{ and comes from }(k+1,n-1)\right).
\end{split}
\end{equation*}
We are interested in deducing a law for $\gamma(k,n)=\alpha(k,n)+\beta(k,n)$, which is the probability of the particle being in $(k,n)$. It is well known that $\gamma$ satisfies
\begin{equation}
\frac{1}{c^2}\gamma_{tt}+\frac{k}{c^2}\gamma_t=\gamma_{xx},
\end{equation}
where $c=\frac{\Delta x}{\Delta t}$ and $k=\lim_{\Delta t\rightarrow0}\frac{2-2p}{\Delta t}\geq0$. Reorganizing and denoting $\gamma=u$ we are left with.
\begin{equation*}
u_{tt}+ku_t=c^2u_{xx},
\end{equation*}
which is precisely the problem we have been studying so far.

Let us now consider a financial asset, a stock share, for example, whose price resembles the motion of the described particle, i.e., it shows certain tendency to repeat the movement previously done.
Then, given suitable initial conditions this equation models the probability density function of the price of the asset, a random variable. We are going to study  the problem given by
\begin{equation}\label{ap2,d,d'}
\begin{array}{l}\left\{\begin{array}{cc}u_{tt}+ku_t=c^2u_{xx}\\u(0)=\delta, \quad u_t(0)=-c\delta '&\end{array}\right.\\\end{array}
\end{equation}
The first initial condition $u(x,0)=\delta(x)$ means that we are absolutely sure the particle, in this case the price of the asset, is 0 at time $t=0$. The second one $u_t(x,0)=-c\delta'(x)$ indicates us the particle, or the price of the asset, has an initial tendency to go upwards, to increase its value.

As a matter of fact, let us explain $u_t(x,0)$. We assume an initial tendency to go upwards, that is, at time $\Delta t$ we know the particle is in $\Delta x$, so we also have $u(x,\Delta t)=\delta (x-\Delta x)$. By definition,
\begin{equation*}
u_t(x,0)=\lim_{\Delta t\rightarrow0}\frac{u(x,\Delta x)-u(x,0)}{\Delta t}=\lim_{\Delta t\rightarrow0}\frac{\delta(x-\Delta x)-\delta(x)}{\Delta t},
\end{equation*}
and since we are dealing with distributions, we compute its distributional derivative:
\begin{equation*}
\begin{split}
\left\langle \frac{\delta(x-\Delta x)-\delta(x)}{\Delta t},\varphi(x)\right\rangle &=\int_{\mathbb{R}}\frac{\delta(x-\Delta x)-\delta(x)}{\Delta t}\varphi(x)dx \\
&=\frac{\varphi(\Delta x)-\varphi(0)}{\Delta t}\xrightarrow[{\Delta t\rightarrow0}]{\Delta x=c\Delta t}c\varphi '(0)=:\left\langle -c\delta',\varphi \right\rangle,
\end{split}
\end{equation*}
for all $\varphi\in\mathcal{C}^\infty_0(\mathbb{R})$ and so we shall have $u_t(x,0)=-c\delta'(x)$.
\begin{remark}
This is a well-known price evolution model thoroughly studied in \cite{Kol}. Here we present an alternative construction of the solution using the results we have found for the damped wave equation with general distributions, in this case $(f,g)=(\delta,-c\delta')$, as its initial values.
\end{remark}
In order to solve problem (\ref{ap2,d,d'}), we can use the formula we found out for the complete problem applying the considered initial conditions. Let us remember the general solution of the problem when the initial conditions are distributions is
\begin{equation*}
\left\langle u,\varphi\right\rangle :=e^{-\frac{k}{2}t}\left( \left\langle f,\left( \psi_t+\frac{k}{2}\psi\right) \ast\varphi \right\rangle +\left\langle g,\psi\ast\varphi \right\rangle\right).
\end{equation*}
In our case, we have $f=\delta$ and $g=-c\delta'$ and so we have the following
\begin{theorem}
The solution to the asset price problem (\ref{ap2,d,d'}) is a probability density function and it is given by
\begin{equation*}
\begin{split}
u(x,t)=e^{\frac{-k}{2}t}\delta(x-ct) 
+e^{\frac{-k}{2}t}\left( \alpha(x+ct)\frac{I_0'(2\alpha\sqrt{\lambda_0})}{\sqrt{\lambda_0}}+\frac{k}{4c}I_0(2\alpha\sqrt{\lambda_0})\right) \mathcal{X}_{(-ct,ct)}(x),
\end{split}
\end{equation*}
where $\lambda_0=c^2t^2-x^2$.
\end{theorem}
\begin{remark}
Let us notice that this probability density function is of a mixed type: it has a discrete part governed by the Dirac delta $\delta(x-ct)$ and a continuous part supported in the interval $[-ct,ct]$.
\end{remark}
\subsection*{Acknowledgements}
The author wishes to thank Prof. J.Sol\`{a}-Morales, who was the advisor of his Bachelor's Degree Thesis \cite{Nua} and of the definitive form of the present manuscript. The author also thanks Prof. X. Cabr\'{e} for his valuable comments and suggestions.
This work has been partially supported by the Collaboration Grants Program of the Ministerio de Educaci\'{o}n y Formaci\'{o}n Profesional, Spain (2019-2020)

\end{document}